\documentclass[11pt]{amsart}

\usepackage{amsmath,latexsym,amssymb,amsthm,array,amsfonts}

\usepackage{mathrsfs,dsfont,bbm}

\numberwithin{equation}{section}

\newtheorem{definition}{Definition}[section]
\newtheorem{theorem}{Theorem}[section]

\newtheorem{corollary}{Corollary}[section]
\newtheorem{proposition}{Proposition}[section]

\newtheorem{remark}{Remark}[section]

\begin{document}
\title{Representation theorem for generators of BSDEs driven by $G$-Brownian motion
and its applications}

\footnote[0]{${}^{*}$M. Hu acknowledges the financial support from the National Natural Science
Foundation of China (11201262 and 11101242). \\
K. He acknowledges the financial support from the National Natural Science
Foundation of China (Grant No. 10971220).}

\author[K. He and M. Hu]{Kun He%${}^{\dag,1}$ 
and Mingshang Hu${}^2$}

%\footnote[0]{${}^{\dag}$Corresponding Author
%(hekun\symbol{64}dhu.edu.cn)}

\date{}

\subjclass[2000]{60H07, 60H20}
\maketitle
\begin{center}
{\footnotesize {\it 1. Department of Mathematics, Donghua University\\
2999 North Renmin Rd., Songjiang, Shanghai 201620, P.R. China\\
2. School of Mathematics, Shandong University\\27 Shanda Nanlu, Jinan {250100}, P.R. China}}
\end{center}
%\maketitle

{\bf Abstract:}
We obtain a representation theorem for the generators of BSDEs driven by G-Brownian motions, and then we use the representation theorem to get a converse comparison theorem
 for G-BSDEs and some equivalent results for nonlinear expectations generated by G-BSDEs.

% \keywords{}
{\bf keywords: }{G-Brownian motion, BSDEs, representation theorem}
\section{Introduction}

 \quad ~~ In 1997, Peng (\cite{Peng1997}) introduced $g-$expectations basing on Backward Stochastic Differential Equations (BSDEs) (\cite{PP90}).
One of the important properties of $g-$expectations is comparison theorem or monotonicity. Chen (\cite{Chen1998}) first consider a converse
result of BSDEs under equal case. After that Briand et. al. (\cite{BCHMP}) obtained a converse comparison theorem for BSDEs under general case. They also derived a representation theorem
for the generator $g$. Following this paper, Jiang (\cite{Jiang2005}) discussed
a more general representation theorem, then in his another paper (\cite{Jiang2005b}) showed a more general converse
comparison theorem. Here the representation theorem is an important method in solving the converse comparison problem and other problems (see Jiang \cite{Jiang2008}).

Recently, Hu et. al. (\cite{HJPS}) proved an existence and uniqueness result on BSDEs driven by $G-$Brownian motions (G-BSDEs), further (in \cite{HJPS1}) they gave
a comparison theorem for G-BSDEs. In this paper we consider the representation theorem for generators of G-BSDEs, and then consider the converse comparison
theorem of G-BSDEs and some equivalent results for nonlinear expectations generated by G-BSDEs. In the following, In Section 2, we review some basic concepts and results about $G-$expectations.
We give the representation theorem of G-BSDEs in Section 3; In Section 4, we consider the applications of representation theorem of G-BSDEs, which contain the converse
comparison theorem and some equivalent results for nonlinear expectations generated by G-BSDEs.

\section{Preliminaries}
We review some basic notions and results of $G$-expectation, the related
spaces of random variables and the backward stochastic differential equations
driven by a $G$-Browninan motion. The readers may refer to \cite{HJPS},
\cite{P07a}, \cite{P07b}, \cite{P08a}, \cite{P08b}, \cite{P10} for more details.

\begin{definition}\label{def2.1}
Let $\Omega$ be a given set and let $\mathcal{H}$ be a vector
lattice of real valued functions defined on $\Omega$, namely $c\in \mathcal{H}$
for each constant $c$ and $|X|\in \mathcal{H}$ if $X\in \mathcal{H}$.
$\mathcal{H}$ is considered as the space of random variables. A sublinear
expectation $\mathbb{\hat{E}}$ on $\mathcal{H}$ is a functional $\mathbb{\hat
{E}}:\mathcal{H}\rightarrow \mathbb{R}$ satisfying the following properties:
for all $X,Y\in \mathcal{H}$, we have
%\begin{description}
\item[(a)] Monotonicity: If $X\geq Y$ then $\mathbb{\hat{E}}[X]\geq
\mathbb{\hat{E}}[Y]$;
\item[(b)] Constant preservation: $\mathbb{\hat{E}}[c]=c$;
\item[(c)] Sub-additivity: $\mathbb{\hat{E}}[X+Y]\leq \mathbb{\hat{E}
}[X]+\mathbb{\hat{E}}[Y]$;

\item[(d)] Positive homogeneity: $\mathbb{\hat{E}}[\lambda X]=\lambda
\mathbb{\hat{E}}[X]$ for each $\lambda \geq 0$.
%\end{description}
$(\Omega,\mathcal{H},\mathbb{\hat{E}})$ is called a sublinear expectation space.
\end{definition}

\begin{definition}
\label{def2.2} Let $X_{1}$ and $X_{2}$ be two $n$-dimensional random vectors
defined respectively in sublinear expectation spaces $(\Omega_{1}%
,\mathcal{H}_{1},\mathbb{\hat{E}}_{1})$ and $(\Omega_{2},\mathcal{H}%
_{2},\mathbb{\hat{E}}_{2})$. They are called identically distributed, denoted
by $X_{1}\overset{d}{=}X_{2}$, if $\mathbb{\hat{E}}_{1}[\varphi(X_{1}%
)]=\mathbb{\hat{E}}_{2}[\varphi(X_{2})]$, for all$\  \varphi \in C_{b.Lip}%
(\mathbb{R}^{n})$, where $C_{b.Lip}(\mathbb{R}^{n})$ denotes the space of
bounded and Lipschitz functions on $\mathbb{R}^{n}$.
\end{definition}

\begin{definition}
\label{def2.3} In a sublinear expectation space $(\Omega,\mathcal{H}%
,\mathbb{\hat{E}})$, a random vector $Y=(Y_{1},\cdot \cdot \cdot,Y_{n})$,
$Y_{i}\in \mathcal{H}$, is said to be independent of another random vector
$X=(X_{1},\cdot \cdot \cdot,X_{m})$, $X_{i}\in \mathcal{H}$ under $\mathbb{\hat
{E}}[\cdot]$, denoted by $Y\bot X$, if for every test function $\varphi \in
C_{b.Lip}(\mathbb{R}^{m}\times \mathbb{R}^{n})$ we have $\mathbb{\hat{E}%
}[\varphi(X,Y)]=\mathbb{\hat{E}}[\mathbb{\hat{E}}[\varphi(x,Y)]_{x=X}]$.
\end{definition}

\begin{definition}
\label{def2.4} ($G$-normal distribution) A $d$-dimensional random vector
$X=(X_{1},\cdot \cdot \cdot,X_{d})$ in a sublinear expectation space
$(\Omega,\mathcal{H},\mathbb{\hat{E}})$ is called $G$-normally distributed if
for each $a,b\geq0$ we have
\[
aX+b\bar{X}\overset{d}{=}\sqrt{a^{2}+b^{2}}X,
\]
where $\bar{X}$ is an independent copy of $X$, i.e., $\bar{X}\overset{d}{=}X$
and $\bar{X}\bot X$. Here the letter $G$ denotes the function
\[
G(A):=\frac{1}{2}\mathbb{\hat{E}}[\langle AX,X\rangle]:\mathbb{S}%
_{d}\rightarrow \mathbb{R},
\]
where $\mathbb{S}_{d}$ denotes the collection of $d\times d$ symmetric matrices.
\end{definition}

Peng \cite{P08b} showed that $X=(X_{1},\cdot \cdot \cdot,X_{d})$ is $G$-normally
distributed if and only if for each $\varphi \in C_{b.Lip}(\mathbb{R}^{d})$,
$u(t,x):=\mathbb{\hat{E}}[\varphi(x+\sqrt{t}X)]$, $(t,x)\in \lbrack
0,\infty)\times \mathbb{R}^{d}$, is the solution of the following $G$-heat
equation:%
\[
\partial_{t}u-G(D_{x}^{2}u)=0,\ u(0,x)=\varphi(x).
\]

The function $G(\cdot):\mathbb{S}_{d}\rightarrow \mathbb{R}$ is a monotonic,
sublinear mapping on $\mathbb{S}_{d}$ and $G(A)=\frac{1}{2}\mathbb{\hat{E}%
}[\langle AX,X\rangle]\leq \frac{1}{2}|A|\mathbb{\hat{E}}[|X|^{2}]$ implies
that there exists a bounded, convex and closed subset $\Gamma \subset
\mathbb{S}_{d}^{+}$ such that
\[
G(A)=\frac{1}{2}\sup_{\gamma \in \Gamma}\mathrm{tr}[\gamma A],
\]
where $\mathbb{S}_{d}^{+}$ denotes the collection of nonnegative elements in
$\mathbb{S}_{d}$.

In this paper, we only consider non-degenerate $G$-normal distribution, i.e.,
there exists some $\underline{\sigma}^{2}>0$ such that $G(A)-G(B)\geq
\underline{\sigma}^{2}\mathrm{tr}[A-B]$ for any $A\geq B$.

\begin{definition}
\label{def2.5} i) Let $\Omega=C_{0}^{d}(\mathbb{R}^{+})$ denote the space of
$\mathbb{R}^{d}$-valued continuous functions on $[0,\infty)$ with $\omega
_{0}=0$ and let $B_{t}(\omega)=\omega_{t}$ be the canonical process. Set
\[
L_{ip}(\Omega):=\{ \varphi(B_{t_{1}},...,B_{t_{n}}):n\geq1,t_{1},...,t_{n}%
\in \lbrack0,\infty),\varphi \in C_{b.Lip}(\mathbb{R}^{d\times n})\}.
\]
Let $G:\mathbb{S}_{d}\rightarrow \mathbb{R}$ be a given monotonic and sublinear
function. $G$-expectation is a sublinear expectation defined by
\[
\mathbb{\hat{E}}[X]=\mathbb{\tilde{E}}[\varphi(\sqrt{t_{1}-t_{0}}\xi_{1}%
,\cdot \cdot \cdot,\sqrt{t_{m}-t_{m-1}}\xi_{m})],
\]
for all $X=\varphi(B_{t_{1}}-B_{t_{0}},B_{t_{2}}-B_{t_{1}},\cdot \cdot
\cdot,B_{t_{m}}-B_{t_{m-1}})$, where $\xi_{1},\cdot \cdot \cdot,\xi_{n}$ are
identically distributed $d$-dimensional $G$-normally distributed random
vectors in a sublinear expectation space $(\tilde{\Omega},\tilde{\mathcal{H}%
},\mathbb{\tilde{E}})$ such that $\xi_{i+1}$ is independent of $(\xi_{1}%
,\cdot \cdot \cdot,\xi_{i})$ for every $i=1,\cdot \cdot \cdot,m-1$. The
corresponding canonical process $B_{t}=(B_{t}^{i})_{i=1}^{d}$ is called a
$G$-Brownian motion.

ii) For each fixed $t\in \lbrack0,\infty)$, the conditional $G$-expectation
$\mathbb{\hat{E}}_{t}$ for $\xi=\varphi(B_{t_{1}}-B_{t_{0}},B_{t_{2}}%
-B_{t_{1}},\cdot \cdot \cdot,B_{t_{m}}-B_{t_{m-1}})\in L_{ip}(\Omega)$, without
loss of generality we suppose $t_{i}=t$, is defined by
\[
\mathbb{\hat{E}}_{t}[\varphi(B_{t_{1}}-B_{t_{0}},B_{t_{2}}-B_{t_{1}}%
,\cdot \cdot \cdot,B_{t_{m}}-B_{t_{m-1}})]
\]%
\[
=\psi(B_{t_{1}}-B_{t_{0}},B_{t_{2}}-B_{t_{1}},\cdot \cdot \cdot,B_{t_{i}%
}-B_{t_{i-1}}),
\]
where
\[
\psi(x_{1},\cdot \cdot \cdot,x_{i})=\mathbb{\hat{E}}[\varphi(x_{1},\cdot
\cdot \cdot,x_{i},B_{t_{i+1}}-B_{t_{i}},\cdot \cdot \cdot,B_{t_{m}}-B_{t_{m-1}%
})].
\]

\end{definition}

For each fixed $T>0$, we set%
\[
L_{ip}(\Omega_{T}):=\{ \varphi(B_{t_{1}},...,B_{t_{n}}):n\geq1,t_{1}%
,...,t_{n}\in \lbrack0,T],\varphi \in C_{b.Lip}(\mathbb{R}^{d\times n})\}.
\]
For each $p\geq1$, we denote by $L_{G}^{p}(\Omega)$ (resp. $L_{G}^{p}%
(\Omega_{T})$) the completion of $L_{ip}(\Omega)$ (resp. $L_{ip}(\Omega_{T})$)
under the norm $\Vert \xi \Vert_{p,G}=(\mathbb{\hat{E}}[|\xi|^{p}])^{1/p}$. It
is easy to check that $L_{G}^{q}(\Omega)\subset L_{G}^{p}(\Omega)$ for $1\leq
p\leq q$ and $\mathbb{\hat{E}}_{t}[\cdot]$ can be extended continuously to
$L_{G}^{1}(\Omega)$. .

For each fixed $\mathbf{a}\in \mathbb{R}^{d}$, $B_{t}^{\mathbf{a}}%
=\langle \mathbf{a},B_{t}\rangle$ is a $1$-dimensional $G_{\mathbf{a}}%
$-Brownian motion, where $G_{\mathbf{a}}(\alpha)=\frac{1}{2}(\sigma
_{\mathbf{aa}^{T}}^{2}\alpha^{+}-\sigma_{-\mathbf{aa}^{T}}^{2}\alpha^{-})$,
$\sigma_{\mathbf{aa}^{T}}^{2}=2G(\mathbf{aa}^{T})$, $\sigma_{-\mathbf{aa}^{T}%
}^{2}=-2G(-\mathbf{aa}^{T})$. Let $\pi_{t}^{N}=\{t_{0}^{N},\cdots,t_{N}^{N}%
\}$, $N=1,2,\cdots$, be a sequence of partitions of $[0,t]$ such that $\mu
(\pi_{t}^{N})=\max \{|t_{i+1}^{N}-t_{i}^{N}|:i=0,\cdots,N-1\} \rightarrow0$,
the quadratic variation process of $B^{\mathbf{a}}$ is defined by%
\[
\langle B^{\mathbf{a}}\rangle_{t}=\lim_{\mu(\pi_{t}^{N})\rightarrow0}%
\sum_{j=0}^{N-1}(B_{t_{j+1}^{N}}^{\mathbf{a}}-B_{t_{j}^{N}}^{\mathbf{a}}%
)^{2}.
\]
For each fixed $\mathbf{a}$, $\mathbf{\bar{a}}\in \mathbb{R}^{d}$, the mutual
variation process of $B^{\mathbf{a}}$ and $B^{\mathbf{\bar{a}}}$ is defined by%
\[
\langle B^{\mathbf{a}},B^{\mathbf{\bar{a}}}\rangle_{t}=\frac{1}{4}[\langle
B^{\mathbf{a}+\mathbf{\bar{a}}}\rangle_{t}-\langle B^{\mathbf{a}%
-\mathbf{\bar{a}}}\rangle_{t}].
\]

\begin{definition}
\label{def2.6} For fixed $T>0$, let $M_{G}^{0}(0,T)$ be the collection of
processes in the following form: for a given partition $\{t_{0},\cdot
\cdot \cdot,t_{N}\}=\pi_{T}$ of $[0,T]$,
\[
\eta_{t}(\omega)=\sum_{j=0}^{N-1}\xi_{j}I_{[t_{j},t_{j+1})}(t),
\]
where $\xi_{j}\in L_{ip}(\Omega_{t_{j}})$, $j=0,1,2,\cdot \cdot \cdot,N-1$. For
$p\geq1$, we denote by $H_{G}^{p}(0,T)$, $M_{G}^{p}(0,T)$ the completion of
$M_{G}^{0}(0,T)$ under the norms $\Vert \eta \Vert_{H_{G}^{p}}=\{ \mathbb{\hat
{E}}[(\int_{0}^{T}|\eta_{s}|^{2}ds)^{p/2}]\}^{1/p}$, $\Vert \eta \Vert
_{M_{G}^{p}}=\{ \mathbb{\hat{E}}[\int_{0}^{T}|\eta_{s}|^{p}ds]\}^{1/p}$ respectively.
\end{definition}

For each $\eta \in M_{G}^{1}(0,T)$, we can define the integrals $\int_{0}%
^{T}\eta_{t}dt$ and $\int_{0}^{T}\eta_{t}d\langle B^{\mathbf{a}}%
,B^{\mathbf{\bar{a}}}\rangle_{t}$ for each $\mathbf{a}$, $\mathbf{\bar{a}}%
\in \mathbb{R}^{d}$. For each $\eta \in H_{G}^{p}(0,T;\mathbb{R}^{d})$ with
$p\geq1$, we can define It\^{o}'s integral $\int_{0}^{T}\eta_{t}dB_{t}$.

Let $S_{G}^{0}(0,T)=\{h(t,B_{t_{1}\wedge t},\cdot \cdot \cdot,B_{t_{n}\wedge
t}):t_{1},\ldots,t_{n}\in \lbrack0,T],h\in C_{b,Lip}(\mathbb{R}^{n+1})\}$. For
$p\geq1$ and $\eta \in S_{G}^{0}(0,T)$, set $\Vert \eta \Vert_{S_{G}^{p}}=\{
\mathbb{\hat{E}}[\sup_{t\in \lbrack0,T]}|\eta_{t}|^{p}]\}^{\frac{1}{p}}$.
Denote by $S_{G}^{p}(0,T)$ the completion of $S_{G}^{0}(0,T)$ under the norm
$\Vert \cdot \Vert_{S_{G}^{p}}$.

We consider the following type of $G$-BSDEs (in this paper we always use
Einstein convention):%
\begin{align}
%\begin{split}
Y_{t}  &  =\xi+\int_{t}^{T}f(s,Y_{s},Z_{s})ds+\int_{t}^{T}g_{ij}(s,Y_{s}%
,Z_{s})d\langle B^{i},B^{j}\rangle_{s}\nonumber \\
&  -\int_{t}^{T}Z_{s}dB_{s}-(K_{T}-K_{t}), \label{pr-eq1}%
%\end{split}
\end{align}%
where%

\[
f(t,\omega,y,z),g_{ij}(t,\omega,y,z):[0,T]\times \Omega_{T}\times
\mathbb{R}\times \mathbb{R}^{d}\rightarrow \mathbb{R}%
\]
satisfy the following properties:

\begin{description}
\item[(H1)] There exists some $\beta>1$ such that for any $y,z$,
$f(\cdot,\cdot,y,z),g_{ij}(\cdot,\cdot,y,z)\in M_{G}^{\beta}(0,T)$.

\item[(H2)] There exists some $L>0$ such that
\[
|f(t,y,z)-f(t,y^{\prime},z^{\prime})|+\sum_{i,j=1}^{d}|g_{ij}(t,y,z)-g_{ij}%
(t,y^{\prime},z^{\prime})|\leq L(|y-y^{\prime}|+|z-z^{\prime}|).
\]

\end{description}

For simplicity, we denote by $\mathfrak{S}_{G}^{\alpha}(0,T)$ the collection
of processes $(Y,Z,K)$ such that $Y\in S_{G}^{\alpha}(0,T)$, $Z\in
H_{G}^{\alpha}(0,T;\mathbb{R}^{d})$, $K$ is a decreasing $G$-martingale with
$K_{0}=0$ and $K_{T}\in L_{G}^{\alpha}(\Omega_{T})$.

\begin{definition}
\label{def3.1} Let $\xi \in L_{G}^{\beta}(\Omega_{T})$ and $f$ satisfy (H1) and
(H2) for some $\beta>1$. A triplet of processes $(Y,Z,K)$ is called a solution
of equation (\ref{pr-eq1}) if for some $1<\alpha \leq \beta$ the following
properties hold:

\begin{description}
\item[(a)] $(Y,Z,K)\in \mathfrak{S}_{G}^{\alpha}(0,T)$;

\item[(b)] $Y_{t}=\xi+\int_{t}^{T}f(s,Y_{s},Z_{s})ds+\int_{t}^{T}%
g_{ij}(s,Y_{s},Z_{s})d\langle B^{i},B^{j}\rangle_{s}-\int_{t}^{T}Z_{s}%
dB_{s}-(K_{T}-K_{t})$.
\end{description}
\end{definition}

\begin{theorem}
\label{the1.1} (\cite{HJPS}) Assume that $\xi \in L_{G}^{\beta}(\Omega_{T})$
and $f$, $g_{ij}$ satisfy (H1) and (H2) for some $\beta>1$. Then equation
(\ref{pr-eq1}) has a unique solution $(Y,Z,K)$. Moreover, for any $1<\alpha<\beta$
we have $Y\in S_{G}^{\alpha}(0,T)$, $Z\in H_{G}^{\alpha}(0,T;\mathbb{R}^{d})$
and $K_{T}\in L_{G}^{\alpha}(\Omega_{T})$.
\end{theorem}

We have the following estimates.

\begin{proposition}
\label{pro3.4} (\cite{HJPS}) Let $\xi \in L_{G}^{\beta}(\Omega_{T})$ and $f$,
$g_{ij}$ satisfy (H1) and (H2) for some $\beta>1$. Assume that $(Y,Z,K)\in
\mathfrak{S}_{G}^{\alpha}(0,T)$ for some $1<\alpha<\beta$ is a solution of
equation (\ref{pr-eq1}). Then there exists a constant $C_{\alpha}>0$ depending on
$\alpha$, $T$, $G$, $L$ such that%
\[
|Y_{t}|^{\alpha}\leq C_{\alpha}\mathbb{\hat{E}}_{t}[|\xi|^{\alpha}+(\int
_{t}^{T}|h_{s}^{0}|ds)^{\alpha}],
\]%
\[
\mathbb{\hat{E}}[(\int_{0}^{T}|Z_{s}|^{2}ds)^{\frac{\alpha}{2}}]\leq
C_{\alpha}\{ \mathbb{\hat{E}}[\sup_{t\in \lbrack0,T]}|Y_{t}|^{\alpha
}]+(\mathbb{\hat{E}}[\sup_{t\in \lbrack0,T]}|Y_{t}|^{\alpha}])^{\frac{1}{2}%
}(\mathbb{\hat{E}}[(\int_{0}^{T}h_{s}^{0}ds)^{\alpha}])^{\frac{1}{2}}\},
\]%
\[
\mathbb{\hat{E}}[|K_{T}|^{\alpha}]\leq C_{\alpha}\{ \mathbb{\hat{E}}%
[\sup_{t\in \lbrack0,T]}|Y_{t}|^{\alpha}]+\mathbb{\hat{E}}[(\int_{0}^{T}%
h_{s}^{0}ds)^{\alpha}]\},
\]
where $h_{s}^{0}=|f(s,0,0)|+\sum_{i,j=1}^{d}|g_{ij}(s,0,0)|$.
\end{proposition}

\begin{proposition}
\label{norm} (\cite{Song11,HJPS}) Let $\alpha \geq1$ and $\delta>0$ be fixed.
Then there exists a constant $C$ depending on $\alpha$ and $\delta$ such that%
\[
\mathbb{\hat{E}}[\sup_{t\in \lbrack0,T]}\mathbb{\hat{E}}_{t}[|\xi|^{\alpha
}]]\leq C\{(\mathbb{\hat{E}}[|\xi|^{\alpha+\delta}])^{\alpha/(\alpha+\delta
)}+\mathbb{\hat{E}}[|\xi|^{\alpha+\delta}]\},\  \forall \xi \in L_{G}%
^{\alpha+\delta}(\Omega_{T}).
\]

\end{proposition}

\begin{theorem}
\label{com} (\cite{HJPS1}) Let $(Y^{l},Z^{l},K^{l})$, $l=1,2$, be the
solutions of the following $G$-BSDEs:%
\begin{align*}
Y_{t}^{l}  &  =\xi+\int_{t}^{T}f(s,Y_{s}^{l},Z_{s}^{l})ds+\int_{t}^{T}%
g_{ij}(s,Y_{s}^{l},Z_{s}^{l})d\langle B^{i},B^{j}\rangle_{s}\\
&  +V_{T}^{l}-V_{t}^{l}-\int_{t}^{T}Z_{s}^{l}dB_{s}-(K_{T}^{l}-K_{t}^{l}),
\end{align*}
where $\xi \in L_{G}^{\beta}(\Omega_{T})$, $f$ and $g_{ij}$ satisfy (H1) and
(H2) for some $\beta>1$, $(V_{t}^{l})_{t\leq T}$ are RCLL processes in
$M_{G}^{\beta}(0,T)$ such that $\mathbb{\hat{E}}[\sup_{t\in \lbrack0,T]}%
|V_{t}^{l}|^{\beta}]<\infty$. If $V_{t}^{1}-V_{t}^{2}$ is an increasing
process, then $Y_{t}^{1}\geq Y_{t}^{2}$ for $t\in \lbrack0,T]$.
\end{theorem}

In this paper, we also need the following assumptions for $G$-BSDE
(\ref{pr-eq1}).

\begin{description}
\item[(H3)] For each fixed $(\omega,y,z)\in \Omega_{T}\times \mathbb{R}%
\times \mathbb{R}^{d}$, $t\rightarrow f(t,\omega,y,z)$ and $t\rightarrow
g_{ij}(t,\omega,y,z)$ are continuous.

\item[(H4)] For each fixed $(t,y,z)\in \lbrack0,T)\times \mathbb{R}%
\times \mathbb{R}^{d}$, $f(t,y,z)$, $g_{ij}(t,y,z)\in L_{G}^{\beta}(\Omega
_{t})$ and
\[
\lim_{\varepsilon \rightarrow0+}\frac{1}{\varepsilon}\mathbb{\hat{E}}[\int
_{t}^{t+\varepsilon}(|f(u,y,z)-f(t,y,z)|^{\beta}+\sum_{i,j=1}^{d}%
|g_{ij}(u,y,z)-g_{ij}(t,y,z)|^{\beta})du]=0.
\]

\item[(H5)] For each $(t,\omega,y)\in \lbrack0,T]\times \Omega_{T}%
\times \mathbb{R}$, $f(t,\omega,y,0)=g_{ij}(t,\omega,y,0)=0$.
\end{description}

Assume that $\xi \in L_{G}^{\beta}(\Omega_{T})$, $f$ and $g_{ij}$ satisfy (H1),
(H2) and (H5) for some $\beta>1$. Let $(Y^{T,\xi},Z^{T,\xi},K^{T,\xi})$ be the
solution of $G$-BSDE (\ref{pr-eq1}) corresponding to $\xi$, $f$ and $g_{ij}$
on $[0,T]$. It is easy to check that $Y^{T,\xi}=Y^{T^{\prime},\xi}$ on $[0,T]$
for $T^{\prime}>T$. Following (\cite{HJPS1}), we can define consistent
nonlinear expectation%
\[
\mathbb{\tilde{E}}_{t}[\xi]=Y_{t}^{T,\xi}\text{ for }t\in \lbrack0,T],
\]
and set $\mathbb{\tilde{E}}[\xi]=\mathbb{\tilde{E}}_{0}[\xi]=Y_{0}^{T,\xi}$.

\section{Representation theorem of generators for $G$-BSDEs}

We consider the following type of $G$-FBSDEs:%
\begin{equation}
X_{s}^{t,x}=x+\int_{t}^{s}b(X_{u}^{t,x})du+\int_{t}^{s}h_{ij}(X_{u}%
^{t,x})d\langle B^{i},B^{j}\rangle_{u}+\int_{t}^{s}\sigma(X_{u}^{t,x})dB_{u},
\label{re-eq1}%
\end{equation}%
\begin{align}
\mbox{}^{\varepsilon}Y_{s}^{t,x,y,p}  &  =y+\langle p,X_{t+\varepsilon}%
^{t,x}-x\rangle+\int_{s}^{t+\varepsilon}f(u,\mbox{}^{\varepsilon}%
Y_{u}^{t,x,y,p},\mbox{}^{\varepsilon}Z_{u}^{t,x,y,p})du\nonumber \\
&  +\int_{s}^{t+\varepsilon}g_{ij}(u,\mbox{}^{\varepsilon}Y_{u}^{t,x,y,p}%
,\mbox{}^{\varepsilon}Z_{u}^{t,x,y,p})d\langle B^{i},B^{j}\rangle
_{u}\nonumber \\
&  -\int_{s}^{t+\varepsilon}\mbox{}^{\varepsilon}Z_{u}^{t,x,y,p}%
dB_{u}-(\mbox{}^{\varepsilon}K_{t+\varepsilon}^{t,x,y,p}-\mbox{}^{\varepsilon
}K_{s}^{t,x,y,p}), \label{re-eq2}%
\end{align}
where $h_{ij}=h_{ji}$ and $g_{ij}=g_{ji}$, $1\leq i,j\leq d$.

We now give the main result in this section.

\begin{theorem}
\label{re-the1}Let $b:\mathbb{R}^{n}\rightarrow \mathbb{R}^{n}$, $h_{ij}%
:\mathbb{R}^{n}\rightarrow \mathbb{R}^{n}$ and $\sigma:\mathbb{R}%
^{n}\rightarrow \mathbb{R}^{n\times d}$ be Lipschitz functions and let $f$ and
$g_{ij}$ satisfy (H1), (H2), (H3) and (H4) for some $\beta>1$. Then, for each
$(t,x,y,p)\in \lbrack0,T)\times \mathbb{R}^{n}\times \mathbb{R}\times
\mathbb{R}^{n}$ and $\alpha \in(1,\beta)$, we have%
\begin{align}
&  L_{G}^{\alpha}-\lim_{\varepsilon \rightarrow0+}\frac{1}{\varepsilon}\{
\mbox{}^{\varepsilon}Y_{t}^{t,x,y,p}-y\} \nonumber \\
&  =f(t,y,\sigma^{T}(x)p)+\langle p,b(x)\rangle+2G((g_{ij}(t,y,\sigma
^{T}(x)p)+\langle p,h_{ij}(x)\rangle)_{i,j=1}^{d}) \label{re-eq3}%
\end{align}

\end{theorem}

\begin{proof}
For each fixed $(t,x,y,p)\in \lbrack0,T)\times \mathbb{R}^{n}\times
\mathbb{R}\times \mathbb{R}^{n}$, we write $(Y^{\varepsilon},Z^{\varepsilon
},K^{\varepsilon})$ instead of $(\mbox{}^{\varepsilon}Y^{t,x,y,p}%
,\mbox{}^{\varepsilon}Z^{t,x,y,p},\mbox{}^{\varepsilon}K^{t,x,y,p})$ for
simplicity. We have $\mathbb{\hat{E}}[|X_{t+\varepsilon}^{t,x}|^{\gamma
}]<\infty$ for each $\gamma \geq1$ (see \cite{P10, HJPS1}). Thus, by Theorem
\ref{the1.1}, $G$-BSDE (\ref{re-eq2}) has a unique solution $(Y^{\varepsilon
},Z^{\varepsilon},K^{\varepsilon})$ and $Y_{t}^{\varepsilon}\in L_{G}^{\alpha
}(\Omega_{t})$. We set, for $s\in \lbrack t,t+\varepsilon]$,
\[
\tilde{Y}_{s}^{\varepsilon}=Y_{s}^{\varepsilon}-(y+\langle p,X_{s}%
^{t,x}-x\rangle),\tilde{Z}_{s}^{\varepsilon}=Z_{s}^{\varepsilon}-\sigma
^{T}(X_{s}^{t,x})p\text{ and }\tilde{K}_{s}^{\varepsilon}=K_{s}^{\varepsilon}.
\]
Applying It\^{o}'s formula to $\tilde{Y}_{s}^{\varepsilon}$ on
$[t,t+\varepsilon]$, it is easy to verify that $(\tilde{Y}^{\varepsilon},\tilde
{Z}^{\varepsilon},\tilde{K}^{\varepsilon})$ solves the following $G$-BSDE:%
\begin{align*}
\tilde{Y}_{s}^{\varepsilon} &  =\int_{s}^{t+\varepsilon}f(u,\tilde{Y}%
_{u}^{\varepsilon}+y+\langle p,X_{u}^{t,x}-x\rangle,\tilde{Z}_{u}%
^{\varepsilon}+\sigma^{T}(X_{u}^{t,x})p)du+\int_{s}^{t+\varepsilon}\langle
p,b(X_{u}^{t,x})\rangle du\\
&  +\int_{s}^{t+\varepsilon}g_{ij}(u,\tilde{Y}_{u}^{\varepsilon}+y+\langle
p,X_{u}^{t,x}-x\rangle,\tilde{Z}_{u}^{\varepsilon}+\sigma^{T}(X_{u}%
^{t,x})p)d\langle B^{i},B^{j}\rangle_{u}\\
&  +\int_{s}^{t+\varepsilon}\langle p,h_{ij}(X_{u}^{t,x})\rangle d\langle
B^{i},B^{j}\rangle_{u}-\int_{s}^{t+\varepsilon}\tilde{Z}_{u}^{\varepsilon
}dB_{u}-(\tilde{K}_{t+\varepsilon}^{\varepsilon}-\tilde{K}_{s}^{\varepsilon}).
\end{align*}
From Proposition \ref{pro3.4},
\begin{align*}
|\tilde{Y}_{s}^{\varepsilon}|^{\alpha} &  \leq C_{\alpha}\mathbb{\hat{E}}%
_{s}[(\int_{s}^{t+\varepsilon}(|f(u,y+\langle p,X_{u}^{t,x}-x\rangle
,\sigma^{T}(X_{u}^{t,x})p)|+|\langle p,b(X_{u}^{t,x})\rangle|\\
&  +\sum_{i,j=1}^{d}|g_{ij}(u,y+\langle p,X_{u}^{t,x}-x\rangle,\sigma
^{T}(X_{u}^{t,x})p)|+|\langle p,h_{ij}(X_{u}^{t,x})\rangle|)du)^{\alpha}],
\end{align*}%
and
\begin{align*}
&  \mathbb{\hat{E}}[(\int_{t}^{t+\varepsilon}|\tilde{Z}_{u}^{\varepsilon}%
|^{2}du)^{\alpha/2}]\\
&  \leq C_{\alpha}\{ \mathbb{\hat{E}}[(\int_{t}^{t+\varepsilon}(|f(u,y+\langle
p,X_{u}^{t,x}-x\rangle,\sigma^{T}(X_{u}^{t,x})p)|+|\langle p,b(X_{u}%
^{t,x})\rangle|+|\langle p,h_{ij}(X_{u}^{t,x})\rangle|\\
&  +\sum_{i,j=1}^{d}|g_{ij}(u,y+\langle p,X_{u}^{t,x}-x\rangle,\sigma
^{T}(X_{u}^{t,x})p)|)du)^{\alpha}]+\mathbb{\hat{E}}[\sup_{s\in \lbrack
t,t+\varepsilon]}|\tilde{Y}_{s}^{\varepsilon}|^{\alpha}]\}
\end{align*}
hold for some constant $C_{\alpha}>0$, which only depending
on $\alpha$, $T$, $G$ and $L$.
By Proposition \ref{norm} and the Lipschitz assumption, we obtain %
\begin{align*}
&  \mathbb{\hat{E}}[\sup_{s\in \lbrack t,t+\varepsilon]}|\tilde{Y}%
_{s}^{\varepsilon}|^{\alpha}+(\int_{t}^{t+\varepsilon}|\tilde{Z}%
_{u}^{\varepsilon}|^{2}du)^{\alpha/2}]\\
&  \leq C_{1}\varepsilon^{\alpha}\mathbb{\hat{E}}[1+(\frac{1}{\varepsilon}%
\int_{t}^{t+\varepsilon}(|f(u,0,0)|^{\beta}+\sum_{i,j=1}^{d}|g_{ij}%
(u,0,0)|^{\beta})du)^{\alpha/\beta}+\sup_{s\in \lbrack t,t+\varepsilon]}%
|X_{s}^{t,x}|^{\beta}],
\end{align*}
where  $C_{1}$ is a constant depending on $x$, $y$, $p$, $\alpha$, $\beta$, $T$,
$G$ and $L$. Noting that $\mathbb{\hat{E}}[\sup_{s\in \lbrack t,t+\varepsilon
]}|X_{s}^{t,x}|^{\beta}]\leq C_{2}(1+|x|^{\beta})$ (see \cite{P10, HJPS1}),
where $C_{2}$ depends on $T$ and $L$, and the following inequality holds,
\begin{align*}
  \int_{t}^{t+\varepsilon}(|f(u,0,0)|^{\beta}&+\sum_{i,j=1}^{d}|g_{ij}%
(u,0,0)|^{\beta})du
  \leq2^{\beta-1}\{ \varepsilon(|f(t,0,0)|^{\beta}+\sum_{i,j=1}^{d}%
|g_{ij}(t,0,0)|^{\beta})\\
&  +\int_{t}^{t+\varepsilon}(|f(u,0,0)-f(t,0,0)|^{\beta}+\sum_{i,j=1}%
^{d}|g_{ij}(u,0,0)-g_{ij}(t,0,0)|^{\beta})du\}.
\end{align*}
Together with assumption (H4) we get
\begin{equation}
\mathbb{\hat{E}}[\sup_{s\in \lbrack t,t+\varepsilon]}|\tilde{Y}_{s}%
^{\varepsilon}|^{\alpha}+(\int_{t}^{t+\varepsilon}|\tilde{Z}_{u}^{\varepsilon
}|^{2}du)^{\alpha/2}]\leq C_{3}\varepsilon^{\alpha},\label{re-eq4}%
\end{equation}
where $C_{3}$ depends on $x$, $y$, $p$, $\alpha$, $\beta$, $T$, $G$ and $L$.
Now we prove equation (\ref{re-eq3}). Lets consider%
\begin{align*}
  \frac{1}{\varepsilon}\{Y_{t}^{\varepsilon}-y\}
&  =\frac{1}{\varepsilon}\tilde{Y}_{t}^{\varepsilon}=\frac{1}{\varepsilon
}\mathbb{\hat{E}}_{t}[\tilde{Y}_{t}^{\varepsilon}+\tilde{K}_{t+\varepsilon
}^{\varepsilon}-\tilde{K}_{t}^{\varepsilon}]\\
&  =\frac{1}{\varepsilon}\mathbb{\hat{E}}_{t}[\int_{t}^{t+\varepsilon
}f(u,y+\langle p,X_{u}^{t,x}-x\rangle,\sigma^{T}(X_{u}^{t,x})p)du+\int
_{t}^{t+\varepsilon}\langle p,b(X_{u}^{t,x})\rangle du\\
&  +\int_{t}^{t+\varepsilon}g_{ij}(u,y+\langle p,X_{u}^{t,x}-x\rangle
,\sigma^{T}(X_{u}^{t,x})p)d\langle B^{i},B^{j}\rangle_{u}\\
&  +\int_{t}^{t+\varepsilon}\langle p,h_{ij}(X_{u}^{t,x})\rangle d\langle
B^{i},B^{j}\rangle_{u}]+L_{\varepsilon},
\end{align*}
where
\begin{align*}
L_{\varepsilon}=&\frac{1}{\varepsilon}\left\{ \mathbb{\hat{E}}_{t}[\int
_{t}^{t+\varepsilon}f(u,\tilde{Y}_{u}^{\varepsilon}+y+\langle p,X_{u}%
^{t,x}-x\rangle,\tilde{Z}_{u}^{\varepsilon}+\sigma^{T}(X_{u}^{t,x}%
)p)du+\int_{t}^{t+\varepsilon}\langle p,b(X_{u}^{t,x})\rangle du\right.\\
+&\int
_{t}^{t+\varepsilon}g_{ij}(u,\tilde{Y}_{u}^{\varepsilon}+y+\langle
p,X_{u}^{t,x}-x\rangle,\tilde{Z}_{u}^{\varepsilon}+\sigma^{T}(X_{u}%
^{t,x})p)d\langle B^{i},B^{j}\rangle_{u}+\int_{t}^{t+\varepsilon}\langle
p,h_{ij}(X_{u}^{t,x})\rangle d\langle B^{i},B^{j}\rangle_{u}]\\
-&\mathbb{\hat{E}%
}_{t}[\int_{t}^{t+\varepsilon}f(u,y+\langle p,X_{u}^{t,x}-x\rangle,\sigma
^{T}(X_{u}^{t,x})p)du+\int_{t}^{t+\varepsilon}\langle p,b(X_{u}^{t,x})\rangle
du\\
+&\left.\int_{t}^{t+\varepsilon}g_{ij}(u,y+\langle p,X_{u}^{t,x}-x\rangle
,\sigma^{T}(X_{u}^{t,x})p)d\langle B^{i},B^{j}\rangle_{u}+\int_{t}%
^{t+\varepsilon}\langle p,h_{ij}(X_{u}^{t,x})\rangle d\langle B^{i}%
,B^{j}\rangle_{u}]\right\}.
\end{align*}
 It is easy to check that $|L_{\varepsilon}|\leq
\frac{C_{4}}{\varepsilon}\mathbb{\hat{E}}_{t}[\int_{t}^{t+\varepsilon}%
(|\tilde{Y}_{u}^{\varepsilon}|+|\tilde{Z}_{u}^{\varepsilon}|)du]$, where
$C_{4}$ depends on $G$, $L$ and $T$. Thus by equation (\ref{re-eq4}) we get%
\begin{align*}
\mathbb{\hat{E}}[|L_{\varepsilon}|^{\alpha}] &  \leq \frac{C_{4}^{\alpha}%
}{\varepsilon^{\alpha}}\mathbb{\hat{E}}[(\int_{t}^{t+\varepsilon}(|\tilde
{Y}_{u}^{\varepsilon}|+|\tilde{Z}_{u}^{\varepsilon}|)du)^{\alpha}]\\
&  \leq \frac{2^{\alpha-1}C_{4}^{\alpha}}{\varepsilon^{\alpha}}\mathbb{\hat{E}%
}[(\int_{t}^{t+\varepsilon}|\tilde{Y}_{u}^{\varepsilon}|du)^{\alpha}+(\int
_{t}^{t+\varepsilon}|\tilde{Z}_{u}^{\varepsilon}|du)^{\alpha}]\\
&  \leq2^{\alpha-1}C_{4}^{\alpha}\left\{ \mathbb{\hat{E}}[\sup_{s\in \lbrack
t,t+\varepsilon]}|\tilde{Y}_{s}^{\varepsilon}|^{\alpha}]+\varepsilon
^{-\alpha/2}\mathbb{\hat{E}}[(\int_{t}^{t+\varepsilon}|\tilde{Z}%
_{u}^{\varepsilon}|^{2}du)^{\alpha/2}]\right\} \\
&  \leq2^{\alpha-1}C_{4}^{\alpha}C_{3}(\varepsilon^{\alpha}+\varepsilon
^{\alpha/2}),
\end{align*}
which implies $L_{G}^{\alpha}-\lim_{\varepsilon \rightarrow0+}L_{\varepsilon
}=0$. We set
\begin{align*}
M_{\varepsilon}=&\frac{1}{\varepsilon}\left\{ \mathbb{\hat{E}}_{t}%
[\int_{t}^{t+\varepsilon}f(u,y+\langle p,X_{u}^{t,x}-x\rangle,\sigma^{T}%
(X_{u}^{t,x})p)du+\int_{t}^{t+\varepsilon}\langle p,b(X_{u}^{t,x})\rangle
du\right.\\
+&\int_{t}^{t+\varepsilon}g_{ij}(u,y+\langle p,X_{u}^{t,x}-x\rangle
,\sigma^{T}(X_{u}^{t,x})p)d\langle B^{i},B^{j}\rangle_{u}+\int_{t}%
^{t+\varepsilon}\langle p,h_{ij}(X_{u}^{t,x})\rangle d\langle B^{i}%
,B^{j}\rangle_{u}]\\
-&\mathbb{\hat{E}}_{t}[\int_{t}^{t+\varepsilon}%
f(u,y,\sigma^{T}(x)p)du+\langle p,b(x)\rangle \varepsilon+\int_{t}%
^{t+\varepsilon}g_{ij}(u,y,\sigma^{T}(x)p)d\langle B^{i},B^{j}\rangle_{u}%
\\
+&\left.\int_{t}^{t+\varepsilon}\langle p,h_{ij}(x)\rangle d\langle B^{i}%
,B^{j}\rangle_{u}]\right\}.
\end{align*}
By the Lipschitz condition, we can get
$|M_{\varepsilon}|\leq \frac{C_{5}}{\varepsilon}\mathbb{\hat{E}}_{t}[\int
_{t}^{t+\varepsilon}|X_{u}^{t,x}-x|du]$, where $C_{5}$ depends on $p$, $G$,
$L$ and $T$. Noting that $\mathbb{\hat{E}}[\sup_{s\in \lbrack t,t+\varepsilon
]}|X_{s}^{t,x}-x|^{\alpha}]\leq C_{6}(1+|x|^{\alpha})\varepsilon^{\alpha/2}$
(see \cite{P10, HJPS1}), where $C_{6}$ depends on $L$, $G$ and $\alpha$, thus
we obtain%
\[
\mathbb{\hat{E}}[|M_{\varepsilon}|^{\alpha}]\leq C_{5}^{\alpha}\mathbb{\hat
{E}}[\sup_{s\in \lbrack t,t+\varepsilon]}|X_{s}^{t,x}-x|^{\alpha}]\leq
C_{5}^{\alpha}C_{6}(1+|x|^{\alpha})\varepsilon^{\alpha/2},
\]
which implies $L_{G}^{\alpha}-\lim_{\varepsilon \rightarrow0+}M_{\varepsilon
}=0$. Now we set
\begin{align*}
N_{\varepsilon}=&\frac{1}{\varepsilon}\left\{ \mathbb{\hat{E}}%
_{t}[\int_{t}^{t+\varepsilon}f(u,y,\sigma^{T}(x)p)du+\langle p,b(x)\rangle
\varepsilon+\int_{t}^{t+\varepsilon}g_{ij}(u,y,\sigma^{T}(x)p)d\langle
B^{i},B^{j}\rangle_{u}\right.\\
+&\int_{t}^{t+\varepsilon}\langle p,h_{ij}(x)\rangle
d\langle B^{i},B^{j}\rangle_{u}]-\mathbb{\hat{E}}_{t}[\int_{t}^{t+\varepsilon
}f(t,y,\sigma^{T}(x)p)du+\langle p,b(x)\rangle \varepsilon\\
+&\left.\int_{t}%
^{t+\varepsilon}g_{ij}(t,y,\sigma^{T}(x)p)d\langle B^{i},B^{j}\rangle_{u}%
+\int_{t}^{t+\varepsilon}\langle p,h_{ij}(x)\rangle d\langle B^{i}%
,B^{j}\rangle_{u}]\right\}.
\end{align*}
 It is easy to deduce that $|N_{\varepsilon}|\leq % It is easy to deduce that 这句话换一句，最好能说明一下下面这个式子是用什么方法退出来的。
\frac{C_{7}}{\varepsilon}\mathbb{\hat{E}}_{t}[\int_{t}^{t+\varepsilon
}(|f(u,y,\sigma^{T}(x)p)-f(t,y,\sigma^{T}(x)p)|+\sum_{i,j=1}^{d}%
|g_{ij}(u,y,\sigma^{T}(x)p)-g_{ij}(t,y,\sigma^{T}(x)p)|)du]$, where $C_{7}$
depends on $G$. Then%
\begin{align*}
\mathbb{\hat{E}}[|N_{\varepsilon}|^{\alpha}] &  \leq C_{7}^{\alpha}\frac
{1}{\varepsilon}\mathbb{\hat{E}}[\int_{t}^{t+\varepsilon}(|f(u,y,\sigma
^{T}(x)p)-f(t,y,\sigma^{T}(x)p)|+\sum_{i,j=1}^{d}|g_{ij}(u,y,\sigma
^{T}(x)p)-g_{ij}(t,y,\sigma^{T}(x)p)|)^{\alpha}du]\\
&  \leq C_{7}^{\alpha}(\frac{1}{\varepsilon}\mathbb{\hat{E}}[\int
_{t}^{t+\varepsilon}(|f(u,y,\sigma^{T}(x)p)-f(t,y,\sigma^{T}(x)p)|+\sum
_{i,j=1}^{d}|g_{ij}(u,y,\sigma^{T}(x)p)-g_{ij}(t,y,\sigma^{T}(x)p)|)^{\beta
}du])^{\frac{\alpha}{\beta}}.
\end{align*}
Take limit on both sides of the above inequality and use  assumption (H4), then we have
$$L_{G}^{\alpha}-\lim_{\varepsilon
\rightarrow0+}N_{\varepsilon}=0.$$ On the other hand, since%
\begin{align*}
&  \mathbb{\hat{E}}_{t}[\int_{t}^{t+\varepsilon}f(t,y,\sigma^{T}%
(x)p)du+\langle p,b(x)\rangle \varepsilon+\int_{t}^{t+\varepsilon}%
g_{ij}(t,y,\sigma^{T}(x)p)d\langle B^{i},B^{j}\rangle_{u}+\int_{t}%
^{t+\varepsilon}\langle p,h_{ij}(x)\rangle d\langle B^{i},B^{j}\rangle_{u}]\\
&  =f(t,y,\sigma^{T}(x)p)\varepsilon+\langle p,b(x)\rangle \varepsilon
+\mathbb{\hat{E}}_{t}[(g_{ij}(t,y,\sigma^{T}(x)p)+\langle p,h_{ij}%
(x)\rangle)(\langle B^{i},B^{j}\rangle_{t+\varepsilon}-\langle B^{i}%
,B^{j}\rangle_{t})]\\
&  =(f(t,y,\sigma^{T}(x)p)+\langle p,b(x)\rangle+2G((g_{ij}(t,y,\sigma
^{T}(x)p)+\langle p,h_{ij}(x)\rangle)_{i,j=1}^{d}))\varepsilon.
\end{align*}
Then we have
\begin{align*}
L_{G}^{\alpha}-&\lim_{\varepsilon \rightarrow0+}\frac{1}{\varepsilon
}\{Y_{t}^{\varepsilon}-y\}\\
=&f(t,y,\sigma^{T}(x)p)+\langle p,b(x)\rangle
+2G((g_{ij}(t,y,\sigma^{T}(x)p)+\langle p,h_{ij}(x)\rangle)_{i,j=1}^{d}).
\end{align*} The
proof is complete.
\end{proof}

\section{Some applications}

\subsection{Converse comparison theorem for $G$-BSDEs}

We consider the following $G$-BSDEs:%
\begin{align*}
Y_{t}^{l,\xi}  &  =\xi+\int_{t}^{T}f^{l}(s,Y_{s}^{l,\xi},Z_{s}^{l,\xi}%
)ds+\int_{t}^{T}g_{ij}^{l}(s,Y_{s}^{l,\xi},Z_{s}^{l,\xi})d\langle B^{i}%
,B^{j}\rangle_{s}\\
&  -\int_{t}^{T}Z_{s}^{l,\xi}dB_{s}-(K_{T}^{l,\xi}-K_{t}^{l,\xi}),\ l=1,2,
\end{align*}
where $g_{ij}^{l}=g_{ji}^{l}$.

We first generalized the comparison theorem in \cite{HJPS1}.

\begin{proposition}
\label{con-pro1} Let $f^{l}$ and $g_{ij}^{l}$ satisfy (H1) and (H2) for some
$\beta>1$, $l=1,2$. If $f^{2}-f^{1}+2G((g_{ij}^{2}-g_{ij}^{1})_{i,j=1}%
^{d})\leq0$, then for each $\xi \in L_{G}^{\beta}(\Omega_{T})$, we have
$Y_{t}^{1,\xi}\geq Y_{t}^{2,\xi}$ for $t\in \lbrack0,T]$.
\end{proposition}

\begin{proof}
From the above G-BSDEs, we have%
\begin{align*}
Y_{t}^{2,\xi}    =&~ \xi+\int_{t}^{T}f^{2}(s,Y_{s}^{2,\xi},Z_{s}^{2,\xi}%
)ds+\int_{t}^{T}g_{ij}^{2}(s,Y_{s}^{2,\xi},Z_{s}^{2,\xi})d\langle B^{i}%
,B^{j}\rangle_{s}\\
&  -\int_{t}^{T}Z_{s}^{2,\xi}dB_{s}-(K_{T}^{2,\xi}-K_{t}^{2,\xi})\\
  =& ~\xi+\int_{t}^{T}f^{1}(s,Y_{s}^{2,\xi},Z_{s}^{2,\xi})ds+\int_{t}^{T}%
g_{ij}^{1}(s,Y_{s}^{2,\xi},Z_{s}^{2,\xi})d\langle B^{i},B^{j}\rangle_{s}\\
&  +V_{T}-V_{t}-\int_{t}^{T}Z_{s}^{2,\xi}dB_{s}-(K_{T}^{2,\xi}-K_{t}^{2,\xi}),
\end{align*}
where
\begin{align*}
V_{t}  =&\int_{0}^{t}(f^{2}-f^{1})(s,Y_{s}^{2,\xi},Z_{s}^{2,\xi})ds+\int
_{0}^{t}(g_{ij}^{2}-g_{ij}^{1})(s,Y_{s}^{2,\xi},Z_{s}^{2,\xi})d\langle
B^{i},B^{j}\rangle_{s}\\
  =&\int_{0}^{t}(f^{2}-f^{1}+2G((g_{ij}^{2}-g_{ij}^{1})_{i,j=1}^{d}%
))(s,Y_{s}^{2,\xi},Z_{s}^{2,\xi})ds\\
 & +\int_{0}^{t}(g_{ij}^{2}-g_{ij}^{1})(s,Y_{s}^{2,\xi},Z_{s}^{2,\xi})d\langle
B^{i},B^{j}\rangle_{s}-\int_{0}^{t}2G((g_{ij}^{2}-g_{ij}^{1})_{i,j=1}%
^{d})(s,Y_{s}^{2,\xi},Z_{s}^{2,\xi})ds.
\end{align*}
By the assumption, it is easy to check that $(V_{t})_{t\leq T}$ is a
decreasing process. Thus, using Theorem \ref{com}, we obtain $Y_{t}^{1,\xi}\geq
Y_{t}^{2,\xi}$ for $t\in \lbrack0,T]$.
\end{proof}

\begin{remark}
Suppose $d=1$, and let $f^{1}=10|z|$, $f^{2}=|z|$, $g^{1}=|z|$ and $g^{2}=2|z|$. It
is easy to check that $f^{2}-f^{1}+2G(g^{2}-g^{1})\leq0$. Thus $f^{2}%
-f^{1}+2G((g_{ij}^{2}-g_{ij}^{1})_{i,j=1}^{d})\leq0$ does not imply $f^{2}\leq
f^{1}$ and $(g_{ij}^{2})_{i,j=1}^{d}\leq(g_{ij}^{1})_{i,j=1}^{d}$.
\end{remark}

Now we give the converse comparison theorem.

\begin{theorem}
\label{con-the2} Let $f^{l}$ and $g_{ij}^{l}$ satisfy (H1), (H2), (H3), (H4)
and (H5) for some $\beta>1$, $l=1,2$. If $Y_{t}^{1,\xi}\geq Y_{t}^{2,\xi}$ for
each $t\in \lbrack0,T]$ and $\xi \in L_{G}^{\beta}(\Omega_{T})$, then
$f^{2}-f^{1}+2G((g_{ij}^{2}-g_{ij}^{1})_{i,j=1}^{d})\leq0$ q.s..
\end{theorem}

\begin{proof}
For simplicity, we take the notation $\mathbb{\tilde{E}}_{t}^{l}[\xi
]=Y_{t}^{l,\xi}$, $l=1,2$. For each fixed $(t,y,z)\in \lbrack0,T)\times
\mathbb{R}\times \mathbb{R}^{d}$, lets consider%
\[
\eta_{\varepsilon}=y+\langle z,h_{ij}\rangle(\langle B^{i},B^{j}%
\rangle_{t+\varepsilon}-\langle B^{i},B^{j}\rangle_{t})+\langle
z,B_{t+\varepsilon}-B_{t}\rangle,
\]
where $h_{ij}=h_{ji}\in \mathbb{R}^{d}$. By Theorem \ref{re-the1}, we have, for
each $\alpha \in(1,\beta)$,%
\[
L_{G}^{\alpha}-\lim_{\varepsilon \rightarrow0+}\frac{1}{\varepsilon
}(\mathbb{\tilde{E}}_{t}^{l}[\eta_{\varepsilon}]-y)=f^{l}(t,y,z)+2G((g_{ij}%
^{l}(t,y,z)+\langle z,h_{ij}\rangle)_{i,j=1}^{d}).
\]
Since $\mathbb{\tilde{E}}_{t}^{1}[\eta_{\varepsilon}]\geq \mathbb{\tilde{E}}%
_{t}^{2}[\eta_{\varepsilon}]$, then $$f^{1}(t,y,z)+2G((g_{ij}%
^{1}(t,y,z)+\langle z,h_{ij}\rangle)_{i,j=1}^{d})\geq f^{2}(t,y,z)+2G((g_{ij}%
^{2}(t,y,z)+\langle z,h_{ij}\rangle)_{i,j=1}^{d}) q.s..$$ Take a $h_{ij}$ such
that $\langle z,h_{ij}\rangle=-g_{ij}^{1}(t,y,z)$. Therefore  $\{f^{2}%
-f^{1}+2G((g_{ij}^{2}-g_{ij}^{1})_{i,j=1}^{d})\}(t,y,z)\leq0$ q.s.. By the
assumptions (H2) and (H3), it is easy to deduce that $f^{2}-f^{1}%
+2G((g_{ij}^{2}-g_{ij}^{1})_{i,j=1}^{d})\leq0$ q.s..
\end{proof}

In the following, we use the notation $\mathbb{\tilde{E}}_{t}^{l}[\xi
]=Y_{t}^{l,\xi}$, $l=1,2$.

\begin{corollary}
\label{con-cor3} Let $f^{l}$ and $g_{ij}^{l}$ be deterministic functions and
satisfy (H1), (H2), (H3) and (H5) for some $\beta>1$, $l=1,2$. If
$\mathbb{\tilde{E}}^{1}[\xi]\geq \mathbb{\tilde{E}}^{2}[\xi]$ for each $\xi \in
L_{G}^{\beta}(\Omega_{T})$, then $f^{2}-f^{1}+2G((g_{ij}^{2}-g_{ij}%
^{1})_{i,j=1}^{d})\leq0$.
\end{corollary}

\begin{proof}
Taking $\eta_{\varepsilon}$ as in Theorem \ref{con-the2}, since $f^{l}$ and
$g_{ij}^{l}$ are deterministic, we could get $\mathbb{\tilde{E}}_{t}%
^{l}[\eta_{\varepsilon}]=\mathbb{\tilde{E}}^{l}[\eta_{\varepsilon}]$, for
$l=1$, $2$. And the proof in Theorem \ref{con-the2} still holds true.
\end{proof}

\subsection{Some equivalent relations}

We consider the following $G$-BSDE:%
\begin{align}
Y_{t}  &  =\xi+\int_{t}^{T}f(s,Y_{s},Z_{s})ds+\int_{t}^{T}g_{ij}(s,Y_{s}%
,Z_{s})d\langle B^{i},B^{j}\rangle_{s}\nonumber \\
&  -\int_{t}^{T}Z_{s}dB_{s}-(K_{T}-K_{t}), \label{som-eq1}%
\end{align}
where $g_{ij}=g_{ji}$. We use the notation $\mathbb{\tilde{E}}_{t}[\xi]=Y_{t}$.

\begin{proposition}
\label{som-pro1} Let $f$ and $g_{ij}$ satisfy (H1), (H2), (H3), (H4) and (H5)
for some $\beta>1$ and fix $\alpha \in(1,\beta)$. Then we have

\begin{description}
\item[(1)] $\mathbb{\tilde{E}}_{t}[\xi+\eta]=\mathbb{\tilde{E}}_{t}[\xi]+\eta$
for $t\in \lbrack0,T]$, $\xi \in L_{G}^{\alpha}(\Omega_{T})$ and $\eta \in
L_{G}^{\alpha}(\Omega_{t})$ if and only if for each $t\in \lbrack0,T]$, $y$,
$y^{\prime}\in \mathbb{R}$, $z\in \mathbb{R}^{d}$,
\begin{equation}
f(t,y,z)-f(t,y^{\prime},z)+2G((g_{ij}(t,y,z)-g_{ij}(t,y^{\prime}%
,z))_{i,j=1}^{d})=0; \label{som-eq2}%
\end{equation}

\item[(2)] $\mathbb{\tilde{E}}_{t}[\xi+\eta]\leq \mathbb{\tilde{E}}_{t}%
[\xi]+\mathbb{\tilde{E}}_{t}[\eta]$ for $t\in \lbrack0,T]$, $\xi \in
L_{G}^{\alpha}(\Omega_{T})$ and $\eta \in L_{G}^{\alpha}(\Omega_{T})$ if and
only if for each $t\in \lbrack0,T]$, $y$, $y^{\prime}\in \mathbb{R}$, $z$,
$z^{\prime}\in \mathbb{R}^{d}$,%
\begin{align}
0  &  \geq f(t,y+y^{\prime},z+z^{\prime})-f(t,y,z)-f(t,y^{\prime},z^{\prime
})\nonumber \\
&  +2G((g_{ij}(t,y+y^{\prime},z+z^{\prime})-g_{ij}(t,y,z)-g_{ij}(t,y^{\prime
},z^{\prime}))_{i,j=1}^{d}); \label{som-eq3}%
\end{align}

\item[(3)] $\mathbb{\tilde{E}}_{t}[\lambda \xi+(1-\lambda)\eta]\leq
\lambda \mathbb{\tilde{E}}_{t}[\xi]+(1-\lambda)\mathbb{\tilde{E}}_{t}[\eta]$
for $t\in \lbrack0,T]$, $\lambda \in \lbrack0,1]$, $\xi \in L_{G}^{\alpha}%
(\Omega_{T})$ and $\eta \in L_{G}^{\alpha}(\Omega_{T})$ if and only if for each
$t\in \lbrack0,T]$, $y$, $y^{\prime}\in \mathbb{R}$, $z$, $z^{\prime}%
\in \mathbb{R}^{d}$, $\lambda \in \lbrack0,1]$,%
\begin{align}
0  &  \geq f(t,\lambda y+(1-\lambda)y^{\prime},\lambda z+(1-\lambda)z^{\prime
})-\lambda f(t,y,z)-(1-\lambda)f(t,y^{\prime},z^{\prime})\nonumber \\
&  +2G((g_{ij}(t,\lambda y+(1-\lambda)y^{\prime},\lambda z+(1-\lambda
)z^{\prime})-\lambda g_{ij}(t,y,z)-(1-\lambda)g_{ij}(t,y^{\prime},z^{\prime
}))_{i,j=1}^{d}); \label{som-eq4}%
\end{align}

\item[(4)] $\mathbb{\tilde{E}}_{t}[\lambda \xi]=\lambda \mathbb{\tilde{E}}%
_{t}[\xi]$ for $t\in \lbrack0,T]$, $\lambda \geq0$ and $\xi \in L_{G}^{\alpha
}(\Omega_{T})$ if and only if for each $t\in \lbrack0,T]$, $y\in \mathbb{R}$,
$z\in \mathbb{R}^{d}$, $\lambda \geq0$,%
\begin{equation}
\begin{split}
f(t,\lambda y,\lambda z)-\lambda f(t,y,z)=&2G((\lambda g_{ij}(t,y,z)-g_{ij}%
(t,\lambda y,\lambda z))_{i,j=1}^{d})\\
=&-2G((g_{ij}(t,\lambda y,\lambda
z)-\lambda g_{ij}(t,y,z))_{i,j=1}^{d}).\label{som-eq5}%
\end{split}
\end{equation}
\end{description}
\end{proposition}

\begin{proof}
(1) "$\Rightarrow$" part. For each fixed $t\in \lbrack0,T)$, $y$, $y^{\prime
}\in \mathbb{R}$, $z\in \mathbb{R}^{d}$, we take $$\xi_{\varepsilon}=y+\langle
z,h_{ij}\rangle(\langle B^{i},B^{j}\rangle_{t+\varepsilon}-\langle B^{i}%
,B^{j}\rangle_{t})+\langle z,B_{t+\varepsilon}-B_{t}\rangle~\textmd{and}~\eta=y^{\prime
}-y,$$ where $h_{ij}=h_{ji}\in \mathbb{R}^{d}$. Then, by Theorem \ref{re-the1}
and $\mathbb{\tilde{E}}_{t}[\xi_{\varepsilon}+\eta]=\mathbb{\tilde{E}}_{t}%
[\xi_{\varepsilon}]+\eta$, we can obtain%
\[
f(t,y^{\prime},z)+2G((g_{ij}(t,y^{\prime},z)+\langle z,h_{ij}\rangle
)_{i,j=1}^{d})=f(t,y,z)+2G((g_{ij}(t,y,z)+\langle z,h_{ij}\rangle)_{i,j=1}%
^{d}).
\]
We choose $h_{ij}$ such that $g_{ij}(t,y^{\prime},z)+\langle
z,h_{ij}\rangle=0$, which implies (\ref{som-eq2}).

"$\Longleftarrow$" part. Let $(Y,Z,K)$ be the solution of $G$-BSDE
(\ref{som-eq1}) corresponding to terminal condition $\xi$. We claim that
$(Y_{s}+\eta,Z_{s},K_{s})_{s\in \lbrack t,T]}$ is the solution of $G$-BSDE
(\ref{som-eq1}) corresponding to terminal condition $\xi+\eta$ on $[t,T]$. For
this we only need to check that for $s\in \lbrack t,T]$,
\begin{equation}
\int_{s}^{T}f(u,Y_{u},Z_{u})du+\int_{s}^{T}g_{ij}(u,Y_{u},Z_{u})d\langle
B^{i},B^{j}\rangle_{u}=\int_{s}^{T}f(u,Y_{u}+\eta,Z_{u})du+\int_{s}^{T}%
g_{ij}(u,Y_{u}+\eta,Z_{u})d\langle B^{i},B^{j}\rangle_{u}. \label{som-eq6}%
\end{equation}
By (\ref{som-eq2}) we can get
\begin{align*}
&  \int_{s}^{T}(g_{ij}(u,Y_{u},Z_{u})-g_{ij}(u,Y_{u}+\eta,Z_{u}))d\langle
B^{i},B^{j}\rangle_{u}-2\int_{s}^{T}G((g_{ij}(u,Y_{u},Z_{u})-g_{ij}%
(u,Y_{u}+\eta,Z_{u}))_{i,j=1}^{d})du\\
&  =\int_{s}^{T}(g_{ij}(u,Y_{u},Z_{u})-g_{ij}(u,Y_{u}+\eta,Z_{u}))d\langle
B^{i},B^{j}\rangle_{u}+\int_{s}^{T}(f(u,Y_{u},Z_{u})-f(u,Y_{u}+\eta
,Z_{u}))du\leq0,
\end{align*}
and%
\begin{align*}
&  \int_{s}^{T}(g_{ij}(u,Y_{u}+\eta,Z_{u})-g_{ij}(u,Y_{u},Z_{u}))d\langle
B^{i},B^{j}\rangle_{u}-2\int_{s}^{T}G((g_{ij}(u,Y_{u}+\eta,Z_{u}%
)-g_{ij}(u,Y_{u},Z_{u}))_{i,j=1}^{d})du\\
&  =\int_{s}^{T}(g_{ij}(u,Y_{u}+\eta,Z_{u})-g_{ij}(u,Y_{u},Z_{u}))d\langle
B^{i},B^{j}\rangle_{u}+\int_{s}^{T}(f(u,Y_{u}+\eta,Z_{u})-f(u,Y_{u}%
,Z_{u}))du\leq0,
\end{align*}
which implies (\ref{som-eq6}). The proof of (1) is complete.

(2) "$\Rightarrow$" part. For each fixed $t\in \lbrack0,T)$, $y$, $y^{\prime
}\in \mathbb{R}$, $z$, $z^{\prime}\in \mathbb{R}^{d}$, we consider
$\xi_{\varepsilon}=y+\langle z,h_{ij}\rangle(\langle B^{i},B^{j}%
\rangle_{t+\varepsilon}-\langle B^{i},B^{j}\rangle_{t})+\langle
z,B_{t+\varepsilon}-B_{t}\rangle$ and $\eta_{\varepsilon}=y^{\prime}+\langle
z^{\prime},h_{ij}^{^{\prime}}\rangle(\langle B^{i},B^{j}\rangle_{t+\varepsilon
}-\langle B^{i},B^{j}\rangle_{t})+\langle z^{\prime},B_{t+\varepsilon}%
-B_{t}\rangle$, where $h_{ij}=h_{ji}\in \mathbb{R}^{d}$, $h_{ij}^{\prime
}=h_{ji}^{\prime}\in \mathbb{R}^{d}$. Then, by Theorem \ref{re-the1} and
$\mathbb{\tilde{E}}_{t}[\xi_{\varepsilon}+\eta_{\varepsilon}]=\mathbb{\tilde
{E}}_{t}[\xi_{\varepsilon}]+\mathbb{\tilde{E}}_{t}[\eta_{\varepsilon}]$, we
obtain%
\begin{align*}
& f(t,y+y^{\prime},z+z^{\prime})+2G((g_{ij}(t,y+y^{\prime},z+z^{\prime
})+\langle z,h_{ij}\rangle+\langle z^{\prime},h_{ij}^{\prime}\rangle
)_{i,j=1}^{d})\\
&  \leq f(t,y,z)+f(t,y^{\prime},z^{\prime})+2G((g_{ij}(t,y,z)+\langle
z,h_{ij}\rangle)_{i,j=1}^{d})+2G((g_{ij}(t,y^{\prime},z^{\prime})+\langle
z^{\prime},h_{ij}^{\prime}\rangle)_{i,j=1}^{d}).
\end{align*}
We choose $h_{ij}$, $h_{ij}^{\prime}$ such that $g_{ij}(t,y,z)+\langle
z,h_{ij}\rangle=0$ and $g_{ij}(t,y^{\prime},z^{\prime})+\langle z^{\prime
},h_{ij}^{\prime}\rangle=0$, which implies (\ref{som-eq3}).

"$\Longleftarrow$" part. Let $(Y,Z,K)$ and $(Y^{\prime},Z^{\prime},K^{\prime
})$ be the solutions of $G$-BSDE (\ref{som-eq1}) corresponding to terminal
condition $\xi$ and $\eta$ respectively. Then $(Y+Y^{\prime},Z+Z^{\prime},K)$
solves the following $G$-BSDE:%
\begin{align*}
Y_{t}+Y_{t}^{\prime}  &  =\xi+\eta+\int_{t}^{T}f(s,Y_{s}+Y_{s}^{\prime}%
,Z_{s}+Z_{s}^{\prime})ds+\int_{t}^{T}g_{ij}(s,Y_{s}+Y_{s}^{\prime},Z_{s}%
+Z_{s}^{\prime})d\langle B^{i},B^{j}\rangle_{s}\\
&  +V_{T}-V_{t}-\int_{t}^{T}(Z_{s}+Z_{s}^{\prime})dB_{s}-(K_{T}-K_{t}),
\end{align*}
where
\begin{align*}
V_{t}  &  =-K_{t}^{\prime}-\int_{0}^{t}(f(s,Y_{s}+Y_{s}^{\prime},Z_{s}%
+Z_{s}^{\prime})-f(s,Y_{s},Z_{s})-f(s,Y_{s}^{\prime},Z_{s}^{\prime}))ds\\
&  -\int_{0}^{t}(g_{ij}(s,Y_{s}+Y_{s}^{\prime},Z_{s}+Z_{s}^{\prime}%
)-g_{ij}(s,Y_{s},Z_{s})-g_{ij}(s,Y_{s}^{\prime},Z_{s}^{\prime}))d\langle
B^{i},B^{j}\rangle_{s}\\
&  =-K_{t}^{\prime}-\{ \int_{0}^{t}(g_{ij}(s,Y_{s}+Y_{s}^{\prime},Z_{s}%
+Z_{s}^{\prime})-g_{ij}(s,Y_{s},Z_{s})-g_{ij}(s,Y_{s}^{\prime},Z_{s}^{\prime
}))d\langle B^{i},B^{j}\rangle_{s}\\
&  -2\int_{0}^{t}G((g_{ij}(s,Y_{s}+Y_{s}^{\prime},Z_{s}+Z_{s}^{\prime}%
)-g_{ij}(s,Y_{s},Z_{s})-g_{ij}(s,Y_{s}^{\prime},Z_{s}^{\prime}))_{i,j=1}%
^{d})ds\} \\
&  -\int_{0}^{t}\{f(s,Y_{s}+Y_{s}^{\prime},Z_{s}+Z_{s}^{\prime})-f(s,Y_{s}%
,Z_{s})-f(s,Y_{s}^{\prime},Z_{s}^{\prime})\\
&  +2G((g_{ij}(s,Y_{s}+Y_{s}^{\prime},Z_{s}+Z_{s}^{\prime})-g_{ij}%
(s,Y_{s},Z_{s})-g_{ij}(s,Y_{s}^{\prime},Z_{s}^{\prime}))_{i,j=1}^{d})\}ds.
\end{align*}
By (\ref{som-eq3}), it is easy to check that $V_{t}$ is an increasing process.
Then, by Theorem \ref{com}, we can get $\mathbb{\tilde{E}}_{t}[\xi+\eta
]\leq \mathbb{\tilde{E}}_{t}[\xi]+\mathbb{\tilde{E}}_{t}[\eta]$. The proof of
(2) is complete.

Finally, we could prove (3) as in (2) and (4) as in (1).
\end{proof}

\begin{proposition}
\label{som-pro2} We have

\begin{description}
\item[(1)] If $G(A)+G(-A)>0$ for any $A\in \mathbb{S}_{d}$ and $A\neq0$, then
(\ref{som-eq2}) holds if and only if $f$ and $g_{ij}$ are independent of $y$.

\item[(2)] If there exists an $A\in \mathbb{S}_{d}$ with $A\neq0$ such that
$G(A)+G(-A)=0$ and $G(A)\neq0$, then for any fixed $g(t,y,z)$ satisfying
(H1)-(H5), we have $f(t,y,z)=-2G(A)g(t,y,z)$ and $(g_{ij}(t,y,z))_{i,j=1}%
^{d}=g(t,y,z)A$ satisfying (\ref{som-eq2})$.$
\end{description}
\end{proposition}

\begin{proof}
It is easy to verify (2), and we only need to prove (1). If (\ref{som-eq2}) holds, it is
easy to check that $G((g_{ij}(t,y,z)-g_{ij}(t,0,z))_{i,j=1}^{d})+G((g_{ij}%
(t,0,z)-g_{ij}(t,y,z))_{i,j=1}^{d})=0$ hold. Then, from the assumption, we get
$g_{ij}(t,y,z)=g_{ij}(t,0,z)$. Therefore, by (\ref{som-eq2}), we have $f(t,y,z)=f(t,0,z)$,
 which implies  $f$ and $g_{ij}$ are independent of $y$.
The converse part is obvious.
\end{proof}

%\Acknowledgements{M. Hu acknowledges the financial support from the National Natural Science
%Foundation of China (11201262 and 11101242). \\
%K. He acknowledges the financial support from the National Natural Science
%Foundation of China (Grant No. 10971220).}

%{The author is supported by NSF of China (Grant Nos. )}

%    Insert the bibliography data here.

\end{document}